\documentclass[12pt]{amsart}
\usepackage[latin1]{inputenc}
\usepackage{latexsym, amsfonts, amsmath, %esint, bbding ,
                      amscd,bezier}
\usepackage[brazil,portuguese,english]{babel}
\usepackage{amssymb}
\usepackage{mathrsfs}      
\usepackage[usenames,dvipsnames,svgnames,table]{xcolor}

\def\Qed{\ifhmode\unskip\nobreak\fi\quad 
  \ifmmode\square\else$\square$\fi}

\usepackage{ulem}

\newtheorem{proposition}{Proposition}[section]
\newtheorem{theorem}{Theorem}[section]

\newtheorem{lemma}{Lemma}[section]

\newtheorem{corollary}{Corollary}[section]

\newtheorem{claim}{Claim}

\numberwithin{equation}{section}

\newcounter{remark}[section]
\newenvironment{remark}
{\refstepcounter{remark}\medskip\noindent{\sc Remark\ \thesection.\theremark:}}{\medskip}
\newcounter{alphatheo}
\newenvironment{alphatheo}
{\refstepcounter{alphatheo}\medskip\noindent{\bf Theorem\ \Alph{alphatheo}.} \it }{\medskip}

\newcounter{example}

\renewenvironment{proof}{\medskip\noindent{\sc Proof:}}{\medskip}

\newcommand{\R}{\mathbb R}

\newcommand{\N}{\mathbb N}
\newcommand{\Z}{\mathbb Z}

\def\erre{\mathbb R}

\def\ze{\mathbb Z}

\def\F{{\mathcal F}}

\def\L{{\mathcal L}}

\def\S{{\mathcal S}}
\def\U{{\mathcal U}}

\def\ccinf{C^\infty_{c}}

\def\Op{\hbox{\rm Op}}
\def\<{\langle}
\def\>{\rangle}
\def\eps{\varepsilon}

\begin{document}

\title[Fractional Hardy-Sobolev inequalities ]{Fractional Hardy-Sobolev inequalities for canceling  elliptic differential operators}

\author {J. Hounie}
\address{Departamento de Matem\'atica, Universidade Federal de S\~ao Carlos, S\~ao Carlos, SP,
13565-905, Brasil}
\email{hounie@dm.ufscar.br}

\author { T. Picon}
\address{Departamento \!de\! Computa\c{c}\~ao\! e \!Matem\'atica, Universidade de S\~ao Paulo, Ribeir\~ao Preto, SP, 14040-901, Brasil}
\email{picon@ffclrp.usp.br.br}

\thanks{Work supported in part by CNPq and FAPESP}
\subjclass[2010]{47F05 35A23 35B45 35F05 35J48  35S05 }

\keywords{$L^{1}$ estimates, Hardy-Sobolev inequality, elliptic operators, canceling operators,  pseudodifferential operators.}

\begin{abstract}{
Let $A(D)$ be an elliptic homogeneous linear differential operator of order $\nu$ on $\R^{N}$, $N \geq 2$, from a complex vector space E to a complex vector space F. In this paper we show that if $\ell\in\R$ satisfies $0< \ell <N$ and $\ell \leq \nu$, then the estimate 
\begin{equation}\nonumber
\left(\int_{\R^{N}}| (-\Delta)^{(\nu-\ell)/2}u(x)|^{q}|x|^{-N+(N-\ell)q}\,dx\right)^{1/q}\leq C \|A(D)u\|_{L^{1}}
\end{equation} 
holds for every $u \in C_{c}^{\infty}(\R^{N};E)$ and $1\le q<\frac{N}{N-\ell}$ if and only if $A(D)$ is canceling in the sense of V. Schaftingen \cite{VS}. Here  $(-\Delta)^{a/2}u$ is the fractional Laplacian  defined as a Fourier multiplier. This estimate extends,  implies and unifies a series of classical inequalities discussed by P. Bousquet and V. Schaftingen in \cite{BVS}. We also present a local version of the previous inequality for operators with smooth variables coefficients.}
\end{abstract}

\maketitle

\tableofcontents

%%%%%%%%%%%%%%%1111111111111111
\section{Introduction}

%The previous inequality is the full version of the boundedness in $L^{1}$ norm to Hardy operator $Tg(t)=\frac{1}{t}\int_{0}^{t}g(s)ds$  

P. Bousquet and J. Van Schaftingen \cite{BVS} proved that  an elliptic homogeneous differential operators $A(D)$ on $\R^{N}$, $N\ge2$, with constant coefficients  of order $\nu$ from a vector space E to vector space F  satisfies the Hardy-Sobolev inequality
\begin{equation}\label{eq1}
 \left(\int_{\R^{N}}\frac{|D^{\nu-\ell}u(x)|^{q}}{|x|^{N-(N-\ell)q}}\,dx\right)^{1/q} \leq C \int_{\R^{N}}|A(D)u(x)|\,dx, 
\end{equation}
for every $u \in C_{c}^{\infty}(\R^{N};E)$, 
$\ell \in \left\{ 1,...,\min(\nu,N-1)\right\}$ and $1\leq q<N/(N-\ell)$
if and only if $A(D)$  has the canceling property introduced in \cite{VS}, i.e.,
\begin{equation}\label{cancel1}
\bigcap_{\xi \in \R^{N}\backslash \left\{ 0 \right\}}A(\xi)[E]=\left\{  0 \right\},
\end{equation}
where $A(\xi)$ is the symbol of the operador $A(D)$. Operators with the canceling property are characterized by satisfying the Sobolev-Gagliardo-Nirenberg estimate for the $L^{1}$ norm \cite[Theorem 1.1]{VS}. Note that the latter estimate that may be understood as the limiting case of \eqref{eq1} when $q=N/(N-\ell)$ and $\ell=1$.

Applying the previous inequality for $A(D)=D^{\nu}\doteq(D^\alpha)_{|\alpha|=\nu}$ and $\nu<N$, $N\ge2$, they recover the classical Hardy {estimate} in $L^{1}$ norm
\begin{equation}\label{eq2}
\int_{\R^{N}}\frac{|u(x)|}{|x|^{\nu}}dx \lesssim  \int_{\R^{N}}|D^{\nu}u(x)|dx, 
\quad u \in C_{c}^{\infty}(\R^{N};E)
\end{equation}
and the Hardy-Sobolev inequalities
\begin{equation}\label{eq2.1}
\left(\int_{\R^{N}}\frac{|D^{\nu-1}u(x)|^{q}}{|x|^{N-(N-1)q }}dx\right)^{1/q} \lesssim \int_{\R^{N}}|D^{\nu}u(x)|dx, \quad u \in C_{c}^{\infty}(\R^{N};E), 
\end{equation}
$1\leq q<N/(N-1)$.
%A series of previously known inequalities also follow from \eqref{eq1} and \eqref{eq2} such as ...
%A series of elliptic and canceling differential operators have been   also follow from \eqref{eq2.1} and \eqref{eq2} such as ...
%\vskip1cm
However, many questions related to {the inequalities \eqref{eq2} and \eqref{eq2.1}} still remain. 
%\eqref{eq1} and \eqref{eq2}

In this paper we carry further the study of Hardy-Sobolev inequalities and present some new global estimates for elliptic canceling homogeneous differential operators with constant coefficients as well as local estimates for elliptic canceling differential operators with variable coefficients.

We start by looking at the exponent $\ell$ in the Hardy-Sobolev inequality \eqref{eq1} which is linked to the order of the derivatives on the left hand side.
A natural question is to allow  $\ell$ to take non discrete values after the introduction of appropriate fractional derivatives in the left hand side. 
In order to do so we consider positive fractional powers of the Laplacean defined by the multiplier operator 
\[
g=(-\Delta)^{a/2}f\quad\Longleftrightarrow\quad
\widehat{g}(\xi)=|\xi|^{a/2}\widehat{f}(\xi),\quad  f\in\S(\R^N), \quad a\ge0,
\]
which gives integral powers of the Laplacean when $a$ is an even integer and are related
to the Riesz potentials defined for $-N<a<0$. Hence, $(-\Delta)^{a/2}$ maps continuously $\S(\R^N)$    into $L^2(\R^N)\subset\S'(\R^N)$ for  $a\ge0$.

Our first result is the following

\begin{alphatheo}%theorem A
Let $A(D)$ be an elliptic homogeneous linear differential operator of order $\nu$ on $\R^{N}$, $N\ge2$, from $E$ to $F$ and assume that $0<\ell<N$ and $\ell\le\nu$. The estimate
\begin{equation}\label{main4}
\left(\int_{\R^{N}}| (-\Delta)^{(\nu-\ell)/2}u(x)|^{q}|x|^{-N+(N-\ell)q}\,dx\right)^{1/q}\leq C \int_{\R^{N}}|A(D)u(x)|\,dx,
\end{equation} 
holds for every $u \in C_{c}^{\infty}(\R^{N};E)$,  $1\le q<\frac{N}{N-\ell}$ and some $C>0$ if and only if $A(D)$ is canceling .
\end{alphatheo}

The restriction $\ell<N$ in the classical inequality \eqref{eq2}  is linked to the fact that $|x|^{-s}$ is not integrable when  $s=N$ which could potentially blow-up the left hand-side. However we may consider the limiting case $\ell=N$ if $u(x)$ 
vanishes at $x=0$. 
Our second result is a version of \eqref{eq2} for $\ell=N$, namely 

 \begin{alphatheo}\label{teoBB}%theorem B
Let $A(D)$ be an elliptic and canceling homogeneous linear differential operator of order $\nu$ on $\R^{N}$, $N\ge2$, from $E$ to $F$. If $\nu=N$ then
\begin{equation}\label{main2.2a}
\int_{\R^{N}}\frac{|u(x)|}{|x|^{N}}\,dx\leq C \int_{\R^{N}} |A(D)u(x)|\,(1+|\log{|x|}) \,dx, 
\end{equation} 
holds for some $C>0$ and all $u \in C_{c}^{\infty}(\R^{N};E)$ such that $u(0)=0$.
\end{alphatheo} 

Estimate \eqref{main2.2a} is somewhat similar to an estimate given by \cite[Theorem 3.3]{Tri2} where the weight $|x|^{\nu}$ is present on the right-hand side  (this phenomenon was also observed in \cite[Remark 4.20]{TM}).    

 We consider now linear differential operators $A(x,D)$ of order $\nu$ with smooth complex coefficients in $\Omega \subset \R^{N}$, $N\ge2$, from a complex vector space $E$ to a complex vector space $F$. We say that $A(x,D)$  is canceling at $x_{0} \in \Omega$ if its principal part evaluated at $x_{0}$ denoted by $A_{\nu}(x_{0},D)$ is canceling in the sense of \eqref{cancel1}. If this holds for every $x_{0} \in \Omega$ we say that $A(x,D)$ is canceling.  Theorem 3.2 in \cite{HP3} asserts that $A(x,D)$ is elliptic and canceling if and only if for every $x_{0} \in \Omega$ is contained in a ball $B=B(x_{0},r) \subset \Omega$ such that the a priori estimate
\begin{equation}\label{hpmain}
\|D^{\nu-1}u\|_{L^{N/(N-1)}}\leq C \|A(x,D)u\|_{L^{1}}, \quad u \in C_{c}^{\infty}(B;E),
\end{equation}
holds for some $C>0$.

Our third result is the following local version of Hardy-Sobolev inequalities for differential operators with variable coefficients. Its statement involves pseudo-differential operators in the H\"ormander classes 
$\Op S^m_{\rho,\delta}$, introduced in \cite{H1} which contains the class of classical pseudo-differential operators and is now standard (for additional reading on the subject of pseudo-differential operators we refer to \cite{H2} and \cite{Ta}).
As it is the case in \eqref{hpmain}, the nature of these inequalities is local since we are dealing with (not necessarily homogeneous) operators with variable coefficients. 

%\textcolor{red}{\bf No enunciado do Teorema C do jeito que estava, $\ell$ assume valores discretos mas ficava meio esquisito não considerar valores contínuos de $\ell$ depois de ter feito isso no Teorema A. } 

%\textcolor{blue}{\bf Proponho uma versão mais geral (no sentido de membro esquerdo mais geral) mas que não é da forma se e só se):} 

\begin{alphatheo}%Theorem C
Let $A(x,D)$ of order $\nu$ be  as before and assume that $0<\ell<N$ and $\ell\le\nu$. If $A(x,D)$ is elliptic and canceling in $\Omega$  then for every $x_0\in\Omega$, $1\le q <\frac{N}{N-\ell}$, and any properly supported pseudo-differential operator 
$P_{\nu-\ell}(x,D)  \in \Op S_{1,\delta}^{\nu-\ell}(\Omega)$, $0\le\delta<1$,
%$P_{\nu-\ell}$ with principal symbol homogeneous of degree $\nu-\ell$, 
there exists a neighborhood $\U$ of $x_0$ and $C>0$ such that 
\begin{equation}\label{maineq}
\bigg(\int_{\R^{N}}\frac{|P_{\nu-\ell}\,u(x)|^{q}}{|x|^{N-(N-\ell)q}}\,dx\bigg)^{1/q}
\leq C \int_{\R^{N}}|A(x,D)u(x)|\,dx
\end{equation} 
holds for every $u \in C_{c}^{\infty}(\U;E)$.
\end{alphatheo}

When $\nu-\ell\in\{1,\dots,\min(\nu,N-1)\}$ we may take $P_{\nu-\ell}=D^{\nu-\ell}$ and
\eqref{maineq} becomes a local form of \eqref{eq1}. For general $\nu-\ell$ we may choose $P_{\nu-\ell}$ with principal symbol $|\xi |^{\nu-\ell}$ so $P_{\nu-\ell}$ may be regarded as a pseudo-differential approximation of the fractional Laplacean $(-\Delta)^{(\nu-\ell)/2}$ and \eqref{maineq} looks like a local form of \eqref{main4}. 
In Subsection \ref{sec4.2} we will show applications of \eqref{maineq} to some classes of elliptic canceling operators associated to systems of smooth vector fields with complex coefficients.

{ The paper is organized as follows. Section \ref{section2} is devoted to the proof of inequality \eqref{main4} where we present a key  $L^{1}$ estimate that holds when  $A(D)$ has the cancellation property (see Lemma \ref{lemmafrac}). The discussion on the necessity of this condition is presented in Subsection \ref{subsection2.2}. In Theorem \ref{thm2.1} we show that the condition of ellipticity is also necessary when $q$ and $\ell$ belong to a special range. Comments and applications of Theorem A inspired in the inequalities from \cite{BVS} are the subject of Subsection \ref{sec 4}, in particular we show that Theorem A implies the Hardy-Sobolev inequality \eqref{eq1}. The proof of Theorem B, that combines the machinery from Section \ref{section2} with variants of a Hardy inequality with weights, are dealt with in Section \ref{secB}. Finally, in Section 4, we prove Theorem C and discuss some applications for operators $A(x,D)$ with smooth variables coefficients related to elliptic overdetermined systems of vector fields. }

%%%%%%%%%%%%%%%%%%%%%22222222
\section{Fractional Hardy-Sobolev estimates }\label{section2}

This section is devoted to the proof of Theorem A and we will assume throughout that $A(D)$ is elliptic and canceling of order $\nu$ from $E$ to $F$. 

%%%%%%%%%%%%%%%%%%%%ss2.1
\subsection{The cancellation condition is sufficient}
The first steep is to write  
$(-\Delta)^{(\nu-\ell)/2}$ as a composition product with $A(D)$ as one of the factors.
Consider the function $\xi\mapsto H(\xi)\in L(F,E)$ defined by
\[H(\xi)=|\xi|^{\nu-\ell}(A^{\ast}\circ A)^{-1}(\xi)A^{\ast}(\xi) \]
that is smooth in $\R^{N}\backslash \left\{0\right\} $ and homogeneous of 
degree $-\ell$. Here $A^{\ast}(\xi)$ is the symbol of adjoint operator $A^{\ast}(D)$.
Since we are assuming that $0<\ell<N$ then  $H$ is a locally integrable  tempered distribution and its inverse Fourier transform $G(x)$,
\[
G\doteq\F^{-1}H,  \quad \widehat G=H,
\] 
is a locally integrable tempered distribution homogeneous of degree $-N+\ell$ that satisfies 
\begin{equation}\label{decfrac}
(-\Delta)^{(\nu-\ell)/2}u(x)=\int_{\R^{N}}G(x-y)[A(D)u(y)]\,dy, \quad u \in C_{c}^{\infty}(\R^{N};E). 
\end{equation}
Moreover 
\begin{equation} \label{interc}
\bigcap_{x \in \R^{N} \backslash \left\{ 0 \right\}} ker\, G(x) = \bigcap_{\xi \in \R^{N} \backslash \left\{ 0 \right\}} ker\, A^{\ast}(\xi) 
\end{equation}
as  follows from the definition of $H$.

The following lemma will be essential in the proofs of Theorem A (a local version for operators with variable coefficients will be important in the proof of Theorem C). 
It depends in a significant  way on the estimate (\textit{cf. }\cite[Lemma 2.2]{BVS})
\[
\left| \int_{\R^N} \psi(y)\cdot A(D)u(y)\,dy\right|\le 
C\sum_{j=1}^m \int_{\R^N}|A(D)u(y)|\,|y|^j\, |D_j\psi|\,dy \tag{$\dagger$}
\]
where $A(D)$ is as above, $\psi(y)\in\ccinf(\R^N;F)$, and $m$ is a positive integer that depends on $A(D)$. 

\begin{lemma} \label{lemmafrac}
Let $A(D)$ be an elliptic and canceling homogeneous linear differential operator of order 
$\nu$ on $\R^{N}$ from $E$ to $F$ and assume that $0<\ell<N$, $\ell\le\nu$  and 
$1\le q<\frac{N}{N-\ell}$. If $K(x,y)$ is a locally integrable function in 
$\R^N\times\R^N$  satisfying 
\begin{align}
|K(x,y)|&\leq C |x-y|^{\ell-N}, \quad x\neq y, \label{orderk}\\
|\partial_{y}K(x,y)|&\leq C |x-y|^{\ell-N-1}, \quad x\neq y, \label{orderk2}
\end{align}
then the a priori estimate
\begin{align*}
\bigg(\int_{\R^{N}}\bigg|\int_{\R^{N}} K(x,y)&A(D)u(y)dy\bigg|^{q} |x|^{-N+(N-\ell)q}\,dx\bigg)^{1/q}\\
            &\le C\int_{\R^{N}} |A(D)u(x)|\,dx,\quad u \in C_{c}^{\infty}(\R^{N};E),
\end{align*}
holds for some $C>0$.
\end{lemma}

Note that  estimate \eqref{main4} in Theorem A follows right away from identity \eqref{decfrac} and Lemma \ref{lemmafrac} since $K(x,y)\doteq G(x-y)$ satisfies \eqref{orderk} and \eqref{orderk2}. Indeed, $\widehat{G}(\xi)=|\xi|^{-\ell}\left\{ |\xi|^{\nu}  \Gamma(\xi) \right\}$  with $\Gamma(\xi)\doteq (A\circ A^{\ast})^{-1}(\xi)A^{\ast}(\xi) $ is  
homogeneous of degree $-\ell$ and its restriction  to the sphere $S^{N-1}$ is smooth. Thus, $G(x)$ is  homogeneous of degree $\ell-N$, $D G(x)$ is  homogeneous of degree $\ell-N-1$ and their restrictions to $S^{N-1}$ are smooth, so the required estimates are  immediate. To complete the proof of Theorem A we must prove Lemma \ref{lemmafrac} which we do next.

\begin{proof}
Let $\psi \in C_{c}^{\infty}(B_{1/2})$ be a cut-off function such that $\psi \equiv 1$ on $B_{1/4}$ and write $K(x,y)=K_{1}(x,y)+K_{2}(x,y)$ with
\begin{equation}\nonumber
\begin{array}{ll}
\displaystyle{K_{1}(x,y)=\psi\left(\frac{y}{|x|}\right)K(x,0)}.
\end{array}
\end{equation}
In order to obtain the estimate it is enough conclude that 
\begin{align*}
J_{i}\doteq &\left(\int_{\R^{N}}\left| \int_{\R^{N}} K_{i}(x,y)A(D)u(y)dy   \right|^{q} |x|^{-N+(N-\ell)q} \, \,dx\right)^{1/q} \\
            &\lesssim \int_{\R^{N}} |A(D)u(x)|\,dx, \quad i=1,2,
\end{align*}
for every $u \in C_{c}^{\infty}(\R^{N};E)$  and $q \in [1,\frac{N}{N-\ell})$. Thus,
\begin{align*}
J_{1}&=
\left(\int_{\R^{N}}\left| \int_{\R^{N}} \psi\left(\frac{y}{|x|}\right)A(D)u(y)dy\right|^{q}\frac{|K(x,0)|^{q}}{|x|^{N-(N-\ell)q}}\,dx\right)^{1/q}    \\
& \lesssim \displaystyle{\left(\int_{\R^{N}}\left| \int_{B_{|x|/2}} \frac{|y|}{|x|}A(D)u(y)dy\right|^{q}\frac{1}{|x|^{N}}\,dx\right)^{1/q}}   \\
&=\displaystyle{\left(\int_{\R^{N}}\left| \int_{B_{|x|/2}}  |y|A(D)u(y)dy\right|^{q}\frac{1}{|x|^{N+q}}\,dx\right)^{1/q}}.
\end{align*}
To obtain the  inequality we have used $(\dagger)$ to estimate the integral with respect to $dy$ with the choice $\varphi(y)\doteq\psi(y/|x|) f$ where, for fixed $x$, $f$ is a unit vector in $F$ chosen so that %\textcolor{blue}{Tiago: nao falta um modulo do lado esquerdo?}
\[
\left|\int_{\R^{N}} \psi\left(\frac{y}{|x|}\right)f\cdot A(D)u(y)dy \right|=
\left|\int_{\R^{N}} \psi\left(\frac{y}{|x|}\right)A(D)u(y)dy\right|
\]
while  the bound for $|K(x,0)|$ comes from  \eqref{orderk}. Using Minkowski inequality we have 
\begin{align*}
J_{1}& \lesssim \int_{\R^{N}}|y||A(D)u(y)|\left( \int_{\R^{N}\backslash B(0,2|y|)} \frac{1}{|x|^{N+q}}\,dx \right)^{1/q}dy \\
& \lesssim \int_{\R^{N}}|y||A(D)u(y)| {\frac{1}{|y|}dy} \\
& = \|A(D)u\|_{L^{1}} .
\end{align*}

To estimate $J_{2}$ let us analyze the kernel $K_{2}(x,y)$. Since $K_{2}(x,y)=K(x,y)$ for $2|y|>|x|$ then from  \eqref{orderk} we have 
\begin{equation}\nonumber
|K_{2}(x,y)|\leq C |x-y|^{\;\ell-N}, \quad 2|y|>|x|.
\end{equation}
If $|x|>4|y|$ then $K_{2}(x,y)=K(x,y)-K(x,0)$ that implies, using \eqref{orderk2},
\begin{align*}
|K_{2}(x,y)|&\leq |y|\sup_{z \in [0,y]}|\partial_{y}K(x,z)| \\
&\leq C |y| |x-y|^{\;\ell-N-1}\\
&\lesssim |x-y|^{\;\ell-N}.
\end{align*}
A similar estimate holds  in the region $|x|<4|y|<2|x|$ thanks to the identity  
\begin{equation*}
K_{2}(x,y)=\left[1-\psi \left(\frac{y}{|x|}\right)\right]K(x,y)+\psi\left(\frac{y}{|x|}\right)\left[ K(x,y)-K(x,0) \right].
\end{equation*}
Using once more Minkowski's inequality we have
\begin{align*}
J_{2} \leq  \int_{\R^{N}} \left( \int_{\R^{N}}\frac{|K_{2}(x,y)|^{q}}{|x|^{N-(N-\ell)q}}  \,dx \right)^{1/q} |A(D)u(y)|dy.
\end{align*}
Since $q<\frac{N}{N-\ell}$ we have
\begin{align*}
\int_{\R^{N}}\frac{|K_{2}(x,y)|^{q}}{|x|^{N-(N-\ell)q}}\,dx & \leq 
\int_{B_{2|y|}}\frac{1}{|x-y|^{(N-\ell)q}|x|^{N-(N-\ell)q}}\,dx\\
           &\qquad\qquad\qquad\qquad\qquad\qquad+ \int_{\R^{N} \backslash B_{2|y|}}\frac{|y|^{q}}{|x|^{N+q}}\,dx\\
& \lesssim 1,
\end{align*}
so $J_{2} \lesssim \|A(D)u\|_{L^{1}}$ as we wished.
 \Qed
\end{proof}

%%%%%%%%%%%%%%%%%%%%ss2.2
\subsection{The cancellation condition is necessary}\label{subsection2.2}

The proof follows the lines of \cite[Section 3]{BVS}. Consider 
\[
f \in \bigcap_{\xi \in \R^{N}\backslash \left\{ 0 \right\}}A(\xi)[E]\subset F.
\] 
Let $\psi \in S(\R^{N})$ satisfy $\hat{\psi}(\xi)=1$ in $B(0,1)$ so in particular  
$\int \psi =1$ and set
\[
p_{\lambda}(x)=\lambda^{N}\psi( \lambda x)-\frac{1}{\lambda^{N}}\psi\left(\frac{x}{\lambda} \right), \quad \lambda \geq 1.
\]
Clearly $\|p_{\lambda}\|_{L^{1}} \leq 2 \|\psi \|_{L^{1}}$ and 
$\widehat{p}_{\lambda}(\xi)=\widehat\psi(\xi/\lambda)-\widehat\psi(\lambda \xi)=0$ on the ball $|\xi|<1/\lambda$ for each $\lambda\ge1$.
Setting
\[
\widehat{u}_{\lambda}(\xi)=\widehat{p_{\lambda}}(\xi) (A^{\ast}\circ A)^{-1}(\xi)A^{\ast}(\xi)f
\]
we see that
$u_{\lambda} \in S(\R^{N};E)$ because $\widehat{u}_{\lambda}(\xi)$
vanishes on a neighborhood of the origin. Writing $f=A(\xi)(e_\xi)$, $\xi\ne0$, we check that
$A(D)u_{\lambda}=p_{\lambda}f$. 
By a density argument, we may apply \eqref{main4} to $u_{\lambda}$ to get

\begin{equation}\label{main4a}
\left(\int_{\R^{N}}\frac{| (-\Delta)^{(\nu-\ell)/2}u_{\lambda}(x)|^{q}}{|x|^{{N-(N-\ell)q}}}\,dx\right)^{1/q}\, \,dx \leq C \int_{\R^{N}}|p_{\lambda}(x)f|\,dx
\end{equation} 
for some $C>0$, $1\le q <\frac{N}{N-\ell}$, which implies that the right hand side of the inequality is bounded by a constant independent of $\lambda\ge1$. Note that 
\begin{equation}\nonumber
(-\Delta)^{(\nu-\ell)/2}u_{\lambda}(x)= \left( G \ast A(D)u_{\lambda}\right)(x) = \left( G \ast p_{\lambda}\right)(x)f  
\end{equation}
where $G$ was defined at beginning of {Subsection 2.1}, so \eqref{main4a} may be written as
\begin{equation*}
\left(\int_{\R^{N}}| G \ast (p_{\lambda}(x)f)|^{q}|x|^{-N+(N-\ell)q}\,dx\right)^{1/q}\leq C \int_{\R^{N}}|p_{\lambda}(x)f|\,dx.
\end{equation*}
\begin{claim}\label{claim1}
$\displaystyle{\lim_{\lambda \rightarrow \infty} G \ast p_{\lambda} (x)=G(x)}$ for $x \in \R^{N} \backslash \left\{ 0 \right\}$.
\end{claim}
Therefore, assuming  Claim \ref{claim1}, it follows from Fatou's lemma that 
\begin{equation}\label{enrosco}
\left(\int_{\R^{N}}| G(x)f|^{q}|x|^{-N+(N-\ell)q}\,\,dx\right)^{1/q} \lesssim 1  %\leq C \int_{\R^{N}}|p_{\lambda}(x)[f]|dx
\end{equation}
by letting $\lambda \rightarrow \infty$.

\begin{claim}\label{claim2}
$G(x)f=0$ for $x\neq 0$.
\end{claim}
Hence, taking Claim \ref{claim2} for granted and invoking  \eqref{interc} we conclude that $A^{\ast}(\xi)f=0$ for every $\xi \neq 0$. Since 
$f \in \bigcap_{\xi \in \R^{N}\backslash \left\{ 0 \right\}}A(\xi)[E]$ and $A$ is elliptic it follows that $f=0$. This will prove that $A(D)$ is canceling as soon as we prove the two claims above.

The proof of Claim \ref{claim2} is simple: the fact that $G$ is homogeneous of 
degree $-N+\ell$ implies that the integrand in  \eqref{enrosco} is homogeneous of degree $-N$ and the integral cannot be finite unless
\[
\int_{S^{N-1}}|G(\omega)f|^{q}\,d\omega=0 
\]
which is equivalent to $G(\omega)f=0$ for all $\omega \in S^{N-1}$ or to 
$G(x)f=0$ for $x\neq 0$.

To prove  Claim 1 we may adapt the arguments in \cite[Proposition 3.1]{BVS} as we sketch below. Write $G \ast p_{\lambda}(x)-G(x)=J_{1}(\lambda,x)-J_{2}(\lambda,x)$ where
\begin{align*}
J_{1}(\lambda,x)&=\int_{\R^{N}}\left( G(x-y)-G(x) \right)\lambda^{N}\psi(\lambda y)\,dy,\\
J_{2}(\lambda,x)&=\int_{\R^{N}} G(x-y) \lambda^{-N}\psi(\lambda^{-1}y)\,dy.
\end{align*}
Consider first $J_{1}(\lambda,x)$.  
Taking account of the decay of $\psi$ and choosing $N<\alpha<N+1$ we have 
\begin{align*}
|J_{1}(\lambda,x)| & \leq \sup_{y \in \R^{N}} \left\{ {|\lambda y|^{\alpha} |\psi(\lambda y)|} \right\}\,\, \int_{\R^{N}} |G(x-y)-G(x)|\frac{\lambda^{N}}{\lambda^{\alpha}|y|^{\alpha}}\,dy \\
& \lesssim \lambda^{N-\alpha} \int_{\R^{N}} |G(x-y)-G(x)|\frac{1}{|y|^{\alpha}}\,dy.
\end{align*}
To majorize the right hand side we observe that, since  $G(x)$ is homogeneous of degree $-N+\ell$ and smooth off the origin, we have the estimates
\begin{align*}
|G(x-y)-G(x)|&\lesssim \frac{|y|}{|x|^{N-\ell+1}},\quad |y|<|x|/2,\\ 
|G(x-y)-G(x)|&\lesssim \frac{1}{|x-y|^{N-\ell}}+\frac{1}{|x|^{N-\ell}},\quad
|y|\ge |x|/2.
\end{align*}
Hence
\begin{align*}
\int_{B_{|x|/2}} |G(x-y)-G(x)|\,\frac{1}{|y|^{\alpha}}\,dy &\lesssim \frac{1 } {|x|^{N-\ell-1}} \int_{B_{|x|/2}} \frac{1}{|y|^{\alpha-1}}\,dy\\
&\lesssim \frac{1} {|x|^{\alpha-\ell}} 
\end{align*}
and 
\begin{align*} 
\int_{\R^{N} \backslash B_{|x|/2}} |G(x-y)-G(x)|\,\frac{1}{|y|^{\alpha}}\,dy
&\lesssim \int_{\R^{N} \backslash B_{|x|/2}}\frac{1}{|x-y|^{N-\ell}}
\frac{1}{|y|^{\alpha}}\,dy\\
&\qquad+\frac{1}{|x|^{N-\ell}}\int_{\R^{N} \backslash B_{|x|/2}}  \frac{1}{|y|^{\alpha}}\,dy\\
& \lesssim \frac{1}{|x|^{\alpha-\ell}},
\end{align*}
since $\ell <N<\alpha$. Thus
\begin{align*}\nonumber
|J_{1}(\lambda,x)|  & \lesssim \lambda^{N-\alpha} \frac{1}{|x|^{\alpha-\ell}}.
\end{align*}
To handle $J_2$ we choose $\ell<\beta<N$ and get
\begin{align*}
|J_{2}(\lambda,x)| \lesssim \lambda^{\beta-N} \int_{\R^{N}} \frac{1}{|x-y|^{N-\ell}} \frac{1}{|y|^{\beta}}  dy \leq  \lambda^{\beta-N}  \frac{1}{|x|^{\beta - \ell}}.
\end{align*}
We conclude that 
\[
|\left(G \ast p_{\lambda}\right)(x)-G(x)| \leq |J_{1}(\lambda,x)|+|J_{2}(\lambda,x)| \rightarrow 0,\quad x \neq 0,
\]
as $\lambda \rightarrow \infty$.
\Qed

%%%%%%%%%%%%%%%%%%%%%%%%%%%ss2.3
\subsection{The ellipticity is necessary when 
$\mathbf{q}\boldsymbol{>1}$\ %\boldsymbol{<}\mathbf{\frac{N}{N-1}} 
and $\boldsymbol{0<\ell\le1}$}  
It is known (\textit{cf.} \cite[Theorem 1.4]{BVS}) that if
 \eqref{eq1} holds for  $\ell=1$ and some $1<q<N/(N-1)$ then $A(D)$ must be elliptic. Similarly we have
\begin{theorem}\label{thm2.1}%theorem 2.1 
Let $A(D)$ be an elliptic homogeneous linear homogeneous differential operator of order $\nu$ on $\R^{N}$, $N\ge2$, from $E$ to $F$ and let $0<\ell\le1$, $1<q<\frac{N}{N-1}$. The estimate
\begin{equation}\label{2.7}
\left(\int_{\R^{N}}| (-\Delta)^{(\nu-\ell)/2}u(x)|^{q}|x|^{-N+(N-\ell)q}\,dx\right)^{1/q}\leq C \int_{\R^{N}}|A(D)u(x)|\,dx,
\end{equation} 
holds for every $u \in C_{c}^{\infty}(\R^{N};E)$ and some $C>0$ if and only if $A(D)$ is elliptic and canceling.
\end{theorem}
\begin{proof}
The ``if" part follows from Theorem A and the necessity of the cancellation condition was proved in the previous subsection so it is enough to prove that $A(D)$ is elliptic if \eqref{2.7} holds.

\end{proof}
Assume $A(D)$ is not elliptic. There exists $\xi\in\R^N\setminus\{0\}$ and $e\in E$ such that  $A(\xi)e=0$ and we may assume that $\xi=(1,0,\dots,0)$, $e=(1,0,\dots,0)$ (here we identify $E$ with $\R^d$, $d$= dimension of $E$). Hence, writing
$A(D)=\sum_{|\alpha|=\nu} a_\alpha\, D^\alpha$ 
it follows that $a_{\alpha_0} e=0$ for $\alpha_0=(\nu,0,\dots,0)$. In order to violate \eqref{2.7} we will try functions of the form
\[
u_\eps(x)\,e=\phi(x_1/\eps)\psi(x')\,e,
\quad \phi\in\ccinf(\R),\, \psi\in\ccinf(\R^{N-1}),\, 0<\eps<1,
\]
with the notation $x'=(x_2,\dots,x_N)$. By the choice of $u_\eps$, 
$A(D)u_\eps$ does not contain derivatives of $u_\eps$ with respect to $x_1$ 
of order $>\nu-1$.
Writing $A(D)u_\eps(x)=v_\eps(x_1,x')$ we see that the right hand side of \eqref{2.7}
satisfies	
\begin{equation}\label{2.8}
\int_{\R^{N}}|A(D)u(x)|dx=\int_{\R^{N}}|v_\eps(\eps x_1,x')|\, \eps \,dx=O(\eps^{-\nu+2}),
\end{equation} 
for $0<\eps<1$. Let us look at the left hand side of \eqref{2.7} that involves the expression
\begin{align}
(-&\Delta)^{(\nu-\ell)/2}u_\eps(x)=\frac{1}{(2\pi)^{N}} \int_{\R^N}e^{ix\cdot\xi}\,|\xi|^{\nu-\ell}\,
\eps\widehat\phi(\eps\xi_1)\,\widehat\psi(\xi')\,d\xi_{1}d\xi'\notag\\
&\notag=\frac{1}{(2\pi)^{N}} \int_{\R^N}e^{i(x_1/\eps)\xi_1}\,e^{ix'\cdot\xi'}\,
((\xi_1/\eps)^2+|\xi'|^2 )^{(\nu-\ell)/2}\, \widehat\phi(\xi_1)\,\widehat\psi(\xi')\,d\xi_{1}d\xi'\\
&=\frac{\eps^{\ell-\nu}}{(2\pi)^{N}} \int_{\R^N}e^{ix_1\xi_1/\eps}\,e^{ix'\cdot\xi'}\,
((\xi_1)^2+|\eps\xi'|^2 )^{(\nu-\ell)/2}\, \widehat\phi(\xi_1)\,\widehat\psi(\xi')\,d\xi_{1}d\xi'
           \notag\\
&\doteq \eps^{\ell-\nu}w_\eps(x_1/\eps,x').\label{2.9}
\end{align}
Note that $\widehat w_\eps(\xi)=\big((\xi_1)^2+|\eps\xi'|^2 \big)^{(\nu-\ell)/2}
\, \widehat\phi(\xi_1)\,\widehat\psi(\xi')$ 
converges in $L^1(\R^N)$ to 
\[
 \widehat w_0(\xi)=
|\xi_1|^{\nu-\ell}\, \widehat\phi(\xi_1)\,\widehat\psi(\xi')\in L^1(\R^N)
\]
and therefore the sequence $w_\eps(x)$ converges uniformly in $\R^N$ to the continuous bounded function
$w_0(x)$. Let  $1<q<N/(N-1)$. In view of \eqref{2.9} we may write after introducing the change of variables $x\mapsto (\eps x_1,x')$
\begin{equation*}
\int_{\R^{N}}
\frac{| (-\Delta)^{(\nu-\ell)/2}u_\eps(x)|^{q}}{|x|^{N-(N-\ell)q}}\,dx=
\int_{\R^{N}}
\frac{\eps^{(\ell-\nu)q+1}\,\,| w_\eps( x_1,x')|^{q}}{\big((\eps x_1)^2+|x'|^2\big)^{ \frac{N-(N-\ell)q}{2}}  }\,dx.
\end{equation*}
Recalling \eqref{2.7} and keeping in mind \eqref{2.8} and \eqref{2.9} we see that
\[
\int_{\R^{N}}
\frac{| w_\eps( x_1,x')|^{q}}{\big((\eps x_1)^2+|x'|^2\big)^{\frac{N-(N-\ell)q}{2}}}\,dx
=O(\eps^{2q-(\ell+1)})=O(\eps^{2(q-1)}) %\tag{$\clubsuit$}
\]
in particular, for any $R>0$,
\[
\int_{|x|\le R}
\frac{| w_\eps( x_1,x')|^{q}}{\big((\eps x_1)^2+|x'|^2\big)^{\frac{-N+(N-\ell)q}{2}}}\,dx
=O(\eps^{2(q-1)}).
\]
Letting $\eps\searrow0$ we conclude that $w_0(x)\equiv0$ which is false whenever
$\phi,\psi\not\equiv0$.
\Qed
%\vskip2cm

%%%%%%%%%%%%%%%%%%%ss2.4
\subsection{Examples and comments}\label{sec 4}

We start this section by pointing out that Theorem A implies estimate \eqref{eq1} when $\ell \in \left\{ 1,...,\min(\nu,N-1) \right\}$ and 
$1<q<N/(N-\ell)$. A sketch of the proof goes as follows. Formally, 
$D^{\nu-\ell}=T\circ (-\Delta)^{(\nu-\ell)/2}$ where $T=(T_\alpha)_{|\alpha|=N-\ell}$ is the multiplier operator given by
\[
\widehat{T_\alpha f}(\xi)=\frac{\xi^\alpha}{|\xi|^{\nu-\ell}}\widehat{f}(\xi)
\]
for appropriate $f$. Then, $Tf=K*f$, $f\in\S(\erre^{N})$, is a singular integral operator and its kernel satisfies (we refer to \cite[Chapter III]{S1} on this subject)
\begin{enumerate} %[(i)] [(ii)]
\item  $\|Tf\|_{L^{2}}\leq C \|f\|_{L^{2}}$,  $f \in \S(\erre^{N})$;
\item   $\left| \partial^{\alpha}_{x}K(x) \right| \leq A_{\alpha} |x|^{-n-|\alpha|}$, for $x \neq 0$ and $\alpha\in\ze_+^N$.
\end{enumerate}  
We  shall exploit a well known result on the continuity of singular integrals in weighted spaces $L^p(\omega(x)dx)$, where $\omega(x)$ is a weight in the Muckenhoupt class $A_p$, $1<p<\infty$, (see \cite[Chapter V]{S2} for the definition). Namely,  (\textit{cf.} \cite[p.205]{S2})
\begin{theorem}\label{teo4.1}
Let $T$ be a singular integral operator satisfying {\rm(1)} and {\rm(2)} and let $\omega \in A_{p}$ for some 
$1<p<\infty$. There exists $C>0$ such that
\begin{equation}\nonumber
\int_{\erre^N}|Tf(x)|^{p}\omega(x)dx \leq C \int_{\erre^N}|f(x)|^{p}\omega(x)dx, \quad  f \in \S(\erre^{N}).
\end{equation}
\end{theorem}
Since the function $|x|^a$ belongs to $A_p$, $p>1$ when $-n<a<n(p-1)$ 
(\textit{cf.} \cite[p.218]{S2}) it follows that
\begin{equation}\nonumber
\int_{\erre^N}\frac{|Tf(x)|^{p}}{|x|^{N-(N-\ell)q}}\,dx 
\leq C \int_{\erre^N}\frac{|f(x)|^{p}}{|x|^{N-(N-\ell)q}}\,dx, \quad  f \in \S(\erre^{N}),
\end{equation}
which formally gives for $f=(-\Delta)^{(\nu-\ell)/2}u$
\begin{equation}\nonumber
\int_{\erre^N}\frac{|D^{\nu-\ell}u(x)|^{p}}{|x|^{N-(N-\ell)q}}\,dx 
\leq C \int_{\erre^N}\frac{|(-\Delta)^{(\nu-\ell)/2}u(x)|^{p}}{|x|^{N-(N-\ell)q}}\,dx, \quad  f \in \S(\erre^{N}).
\end{equation}
This formal argument may be rigorously justified to show that Theorem A implies Theorem 1.1 in \cite{BVS} for $\ell \in \left\{ 1,...,\min(\nu,N-1)\right\}$ and 
$1<q<N/(N-\ell)$.

{Next we list particular cases of Theorem A for some specific elliptic and canceling homogeneous linear differential operators generalizing  estimates that had been previously considered for integral values of $\ell$  in  \cite{BVS}, \cite{HP3} and \cite{VS}.
%\begin{itemize}
%\item 
\vglue 0.2cm
\noindent \textbf{Hodge-de Rham complex.} Let $du$ and $d^{*}u$ the exterior derivative and co-exterior derivative of $u \in C_{c}^{\infty}(\R^{N}; \Lambda^{k}\R^{N})$ the set of smooth  k-forms with compact support.  Assume k is neither 1 nor $N-1$ then for $0<\ell\leq 1$ we have  
\begin{equation}\nonumber
\int_{\R^{N}} \frac{| (-\Delta)^{(1-\ell)/2}u(x)|}{|x|^{\ell}}\,dx \leq C (\|du\|_{L^{1}} + \|d^{*}u\|_{L^{1}} ), %\quad  u \in C_{c}^{\infty}(\R^{N}, \Lambda^{k}\R^{N})
\end{equation}
for all $u \in C_{c}^{\infty}(\R^{N}; \Lambda^{k}\R^{N})$. The estimate remains valid for $k=1$ if $d^{*}u=0$ and for  $k=N-1$ if $du=0$.\\
%
%\vglue 0.2cm
\noindent \textbf{Laplace-Beltrami operator.} Consider the operator 
\[(d^{*}d, dd^{*}):C_{c}^{\infty}(\R^{N}; \Lambda^{k}\R^{N}) \rightarrow C_{c}^{\infty}(\R^{N}; \Lambda^{k}\R^{N}) \times C_{c}^{\infty}(\R^{N};\Lambda^{k}\R^{N}).
\] 
%given by $A(D)u:=(d^{*}du, dd^{*}u)$. 
Assume $k\in \left\{1,...,N-1 \right\}$ then for $0<\ell\leq 2$ when $N>2$ or  $0<\ell < 2$ when $N=2$ we have
\begin{equation}\nonumber
\int_{\R^{N}} \frac{| (-\Delta)^{(2-\ell)/2}u(x)|}{|x|^{\ell}}\,dx \leq C (\|dd^{*}u\|_{L^{1}} + \|d^{*}du\|_{L^{1}} ), %\quad  u \in C_{c}^{\infty}(\R^{N}, \Lambda^{k}\R^{N})
\end{equation}
for all $u \in C_{c}^{\infty}(\R^{N}; \Lambda^{k}\R^{N})$.  We point out if $k=0,N$ then the operator reduces to Laplacian $\Delta$ that is elliptic but not canceling. \\
\noindent \textbf{Korn-Sobolev-Strauss operator.} Consider the symmetric derivative operator 
\[
D_{s}:C_{c}^{\infty}(\R^{N};\R^{N}) \rightarrow C_{c}^{\infty}(\R^{N};\R^{N(N+1)/2})
\] 
given by $D_{s}u(x):=f(x)$ with 
\[
f_{j,k}(x):={\frac{\partial_{x_j}u_{k}(x)+\partial_{x_k}u_{j}(x)}{2}},
\quad 1\leq j\leq k \leq N.
\]
For $0<\ell\leq 1$ we have
\begin{equation}\nonumber
\int_{\R^{N}} \frac{| (-\Delta)^{(1-\ell)/2}u(x)|}{|x|^{\ell}}\,dx \leq C \| D_{s}u\|_{L^{1}},  \quad  u \in C_{c}^{\infty}(\R^{N}; \R^{N}).
\end{equation}
%for all $u \in C_{c}^{\infty}(\R^{N}, \Lambda^{k}\R^{N})$. 
%
\noindent  \textbf{Maz'ya inequality.}  Consider the operator 
\[
(\Delta, \nabla \text{div}):C_{c}^{\infty}(\R^{N};\R^{N}) \rightarrow C_{c}^{\infty}(\R^{N};\R^{N}) \times C_{c}^{\infty}(\R^{N};\R^{N}). 
\]
Then, for $0<\ell\leq 2$ when $N>2$ or  $0<\ell < 2$ 
when $N=2$, we have
\begin{equation}\nonumber
\int_{\R^{N}} \frac{| (-\Delta)^{(2-\ell)/2}u(x)|}{|x|^{\ell}}\,dx \leq C (\|\Delta u\|_{L^{1}} + \|\nabla \text{div}\,u\|_{L^{1}} ), %\quad  u \in C_{c}^{\infty}(\R^{N}, \Lambda^{k}\R^{N})
\end{equation}
for all $u \in C_{c}^{\infty}(\R^{N}; \Lambda^{k}\R^{N})$.}
%\end{itemize}}

%%%%%%%%%%%%%%%%%%%%%3333333
\section{Endpoint case: $\nu=N$}\label{secB}

This section is devoted to the proof of Theorem B. One of the ingredients is a Hardy type inequality for weighted Lebesgue spaces that we state in a general setting.

\begin{lemma}\label{l2.1}
Let $N\geq 2$. If $\nu>N$ and $u\in C_{c}^{\infty}(\R^N)$ satisfies
$D^{\alpha}u(0)=0$ for $|\alpha|\le\nu-N$ then
\begin{equation}\label{2.2}
\int_{\R^{N}}\frac{|u(x)|}{|x|^\nu}
\,dx\lesssim \int_{\R^{N}} \frac{|\nabla u(x)|}{|x|^{\nu-1}}\,dx.
\end{equation} 
If $\nu=N$ then
\begin{equation}\label{2.6}
\int_{\R^{N}}\frac{|u(x)|}{|x|^N}
\,dx
\lesssim \int_{\R^{N}}|\nabla u(x)|\frac{\big|\log |x|\big|}{|x|^{N-1}}\,dx, \  u \in C_{c}^{\infty}(\R^N),
\end{equation} 
where it is assumed that $u(0)=0$. 
\end{lemma}

\begin{proof}
Note that the vanishing hypotheses on $u$ grant the integrability of both sides of \eqref{l2.1}  so, by a density argument, it is enough to prove the estimate assuming 
that $D^{k}u$ vanishes on a neighborhood of the origin.
Denote by $S=S^{N-1}$ the unit sphere in $\R^N$, use polar coordinates and integrate by parts for $\nu>N$ to get
\begin{align*}
\int_{\R^{N}}\frac{|u(x)|}{|x|^\nu}\,\,dx&=\frac{1}{N-\nu}
  \int_{S}\left(\int_0^\infty |u(r\theta)| \partial_{r}
                       \left(r^{N-\nu}\right)\,dr\right)d\theta\\
&=-\frac{1}{N-\nu}
  \int_{S}\left(\int_0^\infty \partial_{r}(|u(r\theta)|) \,r^{N-\nu}\,dr\right)d\theta
\end{align*}
so
\begin{align*}
\int_{S} d\theta\int_0^\infty 
          &|u|\,r^{N-1-\nu}\,dr\\
&=-\int_{S} d\theta\int_0^\infty 
          (\partial_r|u|)\,r^{N-\nu}+|u|\,(N-1-\nu)r^{N-1-\nu}\, dr.
\end{align*}
Rearranging terms in the last identity and using $|(\partial_r|u|)|\le|\nabla u|$ we derive  
\[
(\nu-N)\int_{\R^{N}}\frac{|u(x)|}{|x|^\nu}\,dx
               \le \int_{\R^{N}}\frac{|\nabla u(x)|}{|x|^{\nu-1}}\,dx
\]
as we wished to prove. 
The proof of estimate \eqref{2.6} follows the same steps keeping in mind that
\begin{align*}
\int_{\R^{N}}\frac{|u(x)|}{|x|^N}\,\,dx&=
  \int_{S}\left(\int_0^\infty |u(r\theta)| \partial_{r}
                       \left| \ln{r}\right|\,dr\right)d\theta\\
&=-
  \int_{S}\left(\int_0^\infty \partial_{r}(|u(r\theta)|) \,|\ln{r}|\,dr\right)d\theta \\
  &\leq  \int_{\R^{N}}|\nabla u(x)|\frac{\big|\log |x|\big|}{|x|^{N-1}}\,dx.
\end{align*}     
To avoid differentiability issues with the function $|u|$ in the proof above, one may define the smooth function $N_\eps(u)=\sqrt{|u|^2+\eps}$, $\eps>0$,  replace $|u|$ by $N_\eps(u)$ in the computations and then let  $\eps\searrow0$. 
\Qed

\end{proof}

\begin{remark}\label{rem2.1}
The a priori estimate 
\begin{equation}\label{2.5}
\int_{\R^{N}}\frac{|u(x)|}{|x|^N}
\,dx\le C \int_{\R^{N}}\frac{|\nabla u(x)|}{|x|^{N-1}}\,dx,\quad u\in C_{c}^{\infty}(\R^N;E) ,
\end{equation} 
fails for any choice of $C>0$.
For instance, if we set $u_A(x)=v(r)e_1$ where $e_1=(1,0,\dots,0)\in\R^{N}$ and $v=v(r)$ is the radial scalar function defined by
\[
v(r)=
\begin{cases}
0&\text{if $0<r\le 1$,}\\
\log r &\text{if $1<r\le A$,}\\
-(\log A)\,(r-A)+\log A &\text{if $A<r\le A+1,$}\\
0&\text{if $A+1<r<\infty.$}
\end{cases}
\]
Then $v(r)$ is compactly supported, continuous and piecewise differentiable.
We have
\[
\int_{\R^N}\frac{|u|}{|x|^N}\,dx=C\int_0^\infty\frac{v(r)}{r}\,dr
           =\frac{1}{2}(\log^2A+ \log A)
\]
and
\[
\int_0^\infty |v'(r)|\,dr=\int_1^A\frac{1}{r}\,dr+\int_A^{A+1}\log A\,dr=2\log A,
\]
showing that \eqref{2.2} cannot hold for all $u_A$ as $A\to \infty$. By regularizing $u_A$ we may also violate  \eqref{2.2} with test functions.
\end{remark}

{The proof of Theorem B follows from combining \eqref{2.6} with the following estimate}

\begin{proposition}\label{prop2.2}
Let $A(D)$ be an elliptic and canceling homogeneous linear differential operator with order $\nu=N$. If $u(0)=0$ then
\begin{equation}\label{main2.2}
\int_{\R^{N}}|\nabla u(x)|\frac{\big|\log |x|\big|}{|x|^{N-1}}\,dx \leq C \int_{\R^{N}} |A(D)u(x)| (1+|\log{|x|}|) \,dx
\end{equation} 
for some $C>0$ and for all  $u \in C_{c}^{\infty}(\R^{N};E)$.
\end{proposition} 

The proof of \eqref{main2.2} that  involves $\omega(x)=\big|\log|x|\big|$ is similar to the proof of Lemma \ref{lemmafrac} and also makes use of $(\dagger)$.

\begin{proof} Reasoning as before we may write 
 \begin{equation} \nonumber %\label{q1.1}
 Du(x)=\int_{\R^{N}}K(x-y)[A(D)u](y)dy,
 \end{equation}
where $K \in C^{\infty}(\R^{N} \backslash \left\{ 0 \right\}; \mathscr{L}(E; \mathscr{L}(\R^{N}; F) ))$ homogeneous of degree $-1$ and satisfies the singular estimate
\begin{equation}\label{k1.1}
|K(x-y)|\leq c |x-y|^{-1}, \quad x\neq y,
\end{equation}
for some constant $c>0$. Following the same decomposition used in   Lemma \ref{lemmafrac} 
it is enough conclude that 
\begin{align*}
J_{j}\doteq &\left(\int_{\R^{N}}\left| \int_{\R^{N}} K_{j}(x,y)A(y,D)u(y)dy   \right||x|^{-N+1} \big|\log|x|\big|\, \,dx\right)\\
            &\lesssim \int_{\R^{N}} |A(D)u(x)| (1+\big|\log|x|\big| )\,dx, 
\end{align*}
for every $u \in C_{c}^{\infty}(\R^{N};E)$, $j=1,2$. Thanks $(\dagger)$ and homogeneity of $K$ it follows that
\begin{align*}
%\begin{array}{ll}
%\displaystyle{\left(\int_{\R^{N}}\left| \int_{\R^{N}} K_{1}(x,y)A(y,D)u(y)dy   \right|^{q}|x|^{-N+(N-\ell)q}\,dx\right)^{1/q}}=\\
J_{1}&\lesssim 
\int_{\R^{N}}\left| \int_{\R^{N}} \psi\left(\frac{y}{|x|}\right)A(D)u(y)dy\right|\frac{|K(x)|}{|x|^{N-1}} \big|\log|x|\big| \,dx    \\
%& \lesssim \displaystyle{ \int_{\R^{N}}\left| \int_{B_{|x|/2}} \frac{|y|}{|x|}A(D)u(y)dy\right| \frac{ |\log |x||}{|x|^{N}}dx }   \\
&=\displaystyle{ \int_{\R^{N}}\left| \int_{B_{|x|/2}}  |y|A(D)u(y)dy\right| \frac{\big|\log|x|\big|}{|x|^{N+1}}\,dx } \\
& \lesssim \int_{\R^{N}}|y||A(D)u(y)| \left(\int_{\R^{N}\backslash B(0,2|y|)} \frac{|\log |x||}{|x|^{N+1}}\,dx \right)dy \\
%& \lesssim \int_{\R^{N}}|y||A(D)u(y)| \left(  \frac{1+|\log 2|y||}{2|y|} \right)dy \\
& \lesssim \int_{\R^{N}}|A(D)u(y)| (1+|\log |y| |) dy. 
\end{align*}
We now estimate $J_{2}$. By Minkowski's inequality
\begin{align*}
J_{2} \leq  \int_{\R^{N}} \left( \int_{\R^{N}}\frac{|K_{2}(x,y)|}{|x|^{N-1}}  \big|\log|x|\big|\,dx \right) |A(D)u(y)|dy
\end{align*}
where $\displaystyle{\int_{\R^{N}}\frac{|K_{2}(x,y)|}{|x|^{N-1}} \big|\log|x|\big| \, dx }$ is decomposed as 
\begin{align*}
\int_{B_{2|y|}}\frac{|K_{2}(x,y)|}{|x|^{N-1}} \big|\log|x|\big| dx + \int_{\R^{N} \backslash B_{2|y|}}\frac{|K_{2}(x,y)|}{|x|^{N-1}} \big|\log|x|\big| \,dx.
%&\lesssim \int_{B_{4|y|}}\frac{|\ln|x-y||}{|x|^{N}}dx + \frac{|\ln|y||}{|y|}\\
\end{align*}
Since $K_{2}(x,y)=K(x-y)$ for $|x|<2|y|$ then from  \eqref{k1.1} and for  $0<\delta<1$
we have
\begin{align*}
\int_{B_{2|y|}}\frac{1}{|x-y|}\frac{\big|\log|x|\big|}{|x|^{N-1}}\,dx &\leq \frac{1}{(1-\delta)|y|}\int_{|x|<\delta |y|} \frac{\big|\log|x|\big|}{|x|^{N-1}}\,dx   \\
&+ \frac{|\log(\delta|y|)|}{\delta^{N-1}|y|^{N-1}} \int_{\delta |y|<|x|<2|y|} \frac{1}{|x-y|}dy\\
& \lesssim |\log|y||.
\end{align*}
If $4|y|<|x|$ then
$$|K_{2}(x,y)|=|K(x-y)-K(x)|\leq \frac{|y|}{|x-y||x|}\lesssim \frac{|y|}{|x|^{2}}$$
that implies
\begin{align*}
\int_{ 4|y|<|x|} |K_{2}(x-y)|\frac{\big|\log|x|\big|}{|x|^{N-1}}\,dx &\lesssim |y|  \left( \int_{\R^{N} \backslash B_{4|y|}} \frac{ |\log |x||}{|x|^{N+1}}\,dx \right)   \\
& \lesssim |y| \left(\frac{1+|\log4|y||}{4|y|}\right) \\
& \lesssim (1+|\log|y||).
\end{align*}
Clearly,
\begin{align*}
\int_{ 2|y|<|x|<4|y| }|K_{2}(x-y)|\frac{\big|\log|x|\big|}{|x|^{N-1}}\,dx \lesssim |\log|y||.
\end{align*}
Thus,
\begin{align*}
J_{2} &\leq  \int_{\R^{N}} \left( \int_{\R^{N}}\frac{|K_{2}(x,y)|}{|x|^{N-1}} \big|\log|x|\big| \,dx \right) |A(D)u(y)|dy \\
& \lesssim \int_{\R^{N}} |A(D)u(y)| (1+|\log|y||) dy. \Qed
\end{align*}
%\Qed
\end{proof}

%%%%%%%%%%%%%%%%%%%%%44444444444444
\section{Local Hardy-Sobolev inequalities for $A(x,D)$}\label{seclocal}

{We start by proving Theorem C.} Of course, there is no loss of generality in assuming that $0\in\Omega$  and $x_0=0$ so we will always do so from now on.
In order to obtain \eqref{maineq} for $x_0=0$, it is sufficient to prove 
\begin{equation}\nonumber
\int_{\R^{N}}\frac{|P_{\nu-\ell}\,u(x)|}{|x|^{\ell}}\,dx\leq C \int_{\R^{N}}|A_{\nu}(x,D)u|\,dx, \quad u \in C_{c}^{\infty}(B;E),
\end{equation}  
i.e., we may replace the operator $A(x,D)$ by its principal part $A_{\nu}(x,D)$. Indeed, the terms of $A(x,D)u$ that contain derivatives of $u$ of order $<\nu$ may be majorized by  an application of the local Gagliardo-Nirenberg estimate \eqref{hpmain}. In fact, we have $\|D^{k}u\|_{L^{1}}\lesssim \|D^{\nu-1}u\|_{L^{1}}$ for $ k \le\nu-1$ and then \eqref{hpmain} implies that  $\|D^{k}u\|_{L^{1}}\lesssim \|A(x,D)u\|_{L^{1}}$ for $k=0,...,\nu-1$ which implies 
\[
\Big\|\sum_{k<\nu}a_{k}(x)D^{k}u(x)\Big\|_{L^{1}}\lesssim \|A(x,D)u\|_{L^{1}}, \quad u \in C_{c}^{\infty}(B;E). 
\]
Hence, \eqref{maineq} will follow from Proposition \ref{prop4.1} below that is stated in terms of the H\"ormander class of pseudo-differential operators $S^m_{\rho,\delta}$ introduced in \cite{H1} which contains the class of classical pseudo-differential operators and is now standard. For additional reading on this subject we refer to \cite{H2} and \cite{Ta}.

\begin{proposition}\label{prop4.1}
Let $A(x,D)$ as before, $0<\ell<N$ and $\ell\le\nu$. If $A(x,D)$ is elliptic and canceling on $\Omega$, then for every $1\le q <\frac{N}{N-\ell}$ and every properly supported pseudo-differential operator
$P_{\nu-\ell}(x,D)  \in OpS_{1,\delta}^{\nu-\ell}(\Omega)$, $0\le\delta<1$,
there exists a neighborhood $\U$ of the origin  and $C>0$ such that 
\begin{equation}\label{main5}
\bigg(\int_{\R^{N}}\frac{|P_{\nu-\ell}\,u(x)|^{q}}{|x|^{N-(N-\ell)q}}\,dx\bigg)^{1/q}
\leq C \int_{\R^{N}}|A(x,D)u(x)|\,dx
\end{equation} 
holds for every  $u \in C_{c}^{\infty}(\U;E)$.
\end{proposition}

The starting point in the proof of \eqref{main5} is the following simple lemma.

\begin{lemma}\label{lemma4.1}
Let $A(x,D)$ be elliptic of order $\nu$, $A_{\nu}(x,D)$ its principal part  and  
let $0<\ell<N$, $\ell\le\nu$. Given 
$P(x,D) \in OpS_{1,\delta}^{\nu-\ell}(\Omega)$, $0\leq \delta<1$, properly supported,  there exist properly supported pseudo-differential operators $Q_{1}(x,D) \in OpS^{-\ell}_{1,\delta}(\Omega)$ and $Q_{2}(x,D) \in OpS^{-\infty}(\Omega)$  such that for all $u \in C^{\infty}(\Omega;E)$
\begin{equation}\label{green1}
P(x,D)u(x)=Q_{1}(x,D)[A_{\nu}(y,D)u(x)]+Q_{2}(x,D)u(x).
\end{equation}    
\end{lemma}
\begin{proof}
Since $\Delta_{A}\doteq A_{\nu}^{*}(x,D)A_{\nu}(x,D)$ is elliptic  then there exist a parametrix $\tilde{q}(x,D) \in S^{-2\nu}_{1,0}(\Omega)$ and a regularizing operator $\tilde{r}(x,D) \in S^{-\infty}(\Omega)$ such that 
\begin{align*}
P(x,D)u&=P(x,D)\tilde{q}(x,D)A^{\ast}_{\nu}(x,D)A_{\nu}(x,D)u+P(x,D)r(x,D)u\\
&\doteq Q_{1}(x,D)[A_{\nu}(x,D)u]+Q_{2}(x,D)u, \quad u \in C^{\infty}(\Omega)
\end{align*}
where $Q_{1}(x,D)\,\doteq \,P(x,D)\tilde{q}(x,D)A_{\nu}^{*}(x,D) \in S^{-\ell}_{1,\delta}(\Omega)$ and $Q_{2}(x,D)\doteq  P(x,D)\tilde{r}(x,D)\in S^{-\infty}(\Omega)$.  
\Qed
\end{proof}

In view of \eqref{green1}, in order  to obtain \eqref{main5}  it is enough to prove that for some ball $B=B(0,\delta)$ and  $C=C(B)>0$ the following estimates hold for 
$ u \in C_{c}^{\infty}(B;E)$:
\begin{equation}\label{4.4}
\bigg(\int_{\R^{N}}\frac{|Q_{2}(x,D)u(x)|^{q}}{|x|^{N-(N-\ell-\alpha)q}}\,dx\bigg)^{1/q}\leq C \int_{\R^{N}}|A_{\nu}(x,D)u|\,dx,
\end{equation} 
\begin{equation}\label{4.5}
\bigg(\int_{\R^{N}}\frac{|Q_{1}(x,D)A_\nu(x,D)u(x)|^{q}}{|x|^{N-(N-\ell-\alpha)q}}\,dx\bigg)^{1/q}\leq C \int_{\R^{N}}|A_{\nu}(x,D)u|\,dx.
\end{equation} 
The proof of \eqref{4.4} is easy because $Q_2(x,D)$ is regularizing. Choose {$\dfrac{N}{N-(N-\ell-1)q}< p<\dfrac{N}{N-(N-\ell)q}$ and then $r$ and $r_{\ell}^{\ast}$  such that
\[
\frac{1}{r}=1-\frac{1}{p}+\frac{\ell}{N}\,,\qquad
1-\frac{1}{N}<\frac{1}{r= \frac{1}{r^{\ast}_{\ell}}+\frac{\ell}{N}}\,.
% =1-\frac{1}{p}+\frac{\ell}{N}
\]
  We get 
\begin{align*}
\int_{\R^{N}}|Q_{2}(x,D)u(x)|^{q}|x|^{-N+(N-\ell)q}&\,dx\leq \\ 
\|Q_{2}(x,D)u\|^{q}_{L^{qp'}(B_{2})}  \,  
&\||x|^{-N+(N-\ell)q}\|_{L^{p}(B_{2})}, \quad (r_{\ell}^{\ast}\doteq p'q)\\
&\le C\,\|u\|^{q}_{L^{r}} \,\delta_{2}^{\frac{[-N+(N-\ell)q]p+N}{p}}\\
&\le C\,\|A_\nu(x,D) u\|^{q}_{L^1} 
\end{align*} 
where we  have used that  $Q_{2}(x,D)$ is bounded from $L^{r}$ to $L^{r_{\ell}^{*}}$ (see, e.g., \cite[Theorem 3.5 ]{AH}). 

 %The proof of estimate \eqref{3.5} is a consequence of following result

\begin{lemma}\label{lemma4.2}%lemma 4.2
Let $A(x,D)$ as before, $A_{\nu}(x,D)$ be its homogeneous part of order $\nu$, 
$1\le q <\frac{N}{N-\ell}$, $0<\ell<N$ and $\ell\le\nu$. If $A(x,D)$ is elliptic and canceling on $\Omega$, then for any properly supported pseudo-differential operators $Q_{1}(x,D) \in OpS^{-\ell}_{1,\delta}(\Omega)$, $0 \leq \delta<1$,
there exists a neighborhood $\U$ of the origin  and $C>0$ such that \eqref{4.5}
holds for every $u \in C_{c}^{\infty}(\U;E)$.
\end{lemma}

The proof of Lemma \ref{lemma4.2} depends on a local analogue of estimate  $(\dagger)$ (stated right before Lemma \ref{lemmafrac})   
valid for elliptic canceling operators with variable coefficients.

\begin{proposition}\label{prop1}
Let $A(x,D)$ elliptic and canceling as before. There exist a ball $B=B(0,r) \subset \Omega$, $m \in \N^*$ and $C>0$ such that
\begin{align}\label{4.6}
\bigg| \int_{\R^{N}}\varphi(x)\cdot&A_{\nu}(x,D)u(x)\,dx \bigg| \\
\notag  &\leq C \sum_{j=1}^{m} \int_{\R^{N}}|A_{\nu}(x,D)u(x)||x|^{j}|D^{j}\varphi(x)|\,dx 
\end{align} 
for every $u \in C_{c}^{\infty}(B;E)$ and $\varphi \in C_{c}^{\infty}(B;F)$.
\end{proposition}
We will assume this result and postpone its  proof until the next section. 

%{\color{red}{\bf[achei melhor incluir ``its proof" na Subseção 3.1 e jogar o conteúdo atual da Subseção 3.1 para a Seção de aplicações]}}

Let us recall that the kernel distribution  $K(x,y)$  of the operator $Q_{1}(x,D) \in S^{-\ell}(\Omega)$
 \begin{equation}\label{q1}
 Q_{1}(x,D)f(x)=\int_{\R^{N}}K(x,y)f(y)dy
 \end{equation}
is smooth off the diagonal $\big\{(x,x) \in \Omega\times\Omega\big\}$ and satisfies the estimates (see, e.g., \cite[Theorem 1.1]{AH}) 
\begin{equation}\label{k1}
|K(x,y)|\leq C |x-y|^{\ell-N}, \quad x\neq y
\end{equation}
and
\begin{equation*}
|\partial_{y}K(x,y)|\leq C |x-y|^{\ell-N-1}, \quad x\neq y
\end{equation*}
for some $C>0$. We may now prove Lemma \ref{lemma4.2} adapting the arguments of 
Lemma \ref{lemmafrac}.

\begin{proof}
Consider  $\psi \in C_{c}^{\infty}(B_{1/2})$ a cut-off function such that $\psi \equiv 1$ on $B_{1/4}$ and  write the kernel  as $K(x,y)=K_{1}(x,y)+K_{2}(x,y)$ where
\begin{equation}\label{4.9}
\begin{array}{ll}
\displaystyle{K_{1}(x,y)=\psi\left(\frac{y}{|x|}\right)K(x,0)}.
\end{array}
\end{equation}
In order to obtain \eqref{4.5} it is sufficient show that 
\begin{align*}
J_{j}\doteq &\left(\int_{\R^{N}}\left| \int_{\R^{N}} K_{j}(x,y)A_{\nu}(y,D)u(y)dy   \right|^{q}|x|^{-N+(N-\ell)q}\,dx\right)^{1/q}\\
            &\lesssim \|A_{\nu}(x,D)u\|_{L^{1}},\qquad j=1,2,
\end{align*}
for every $u \in C_{c}^{\infty}(\U;E)$  and $1\le q <\frac{N}{N-\ell}$.

We have
\begin{align*}
J_{1}&\lesssim 
\left(\int_{\R^{N}}\left| \int_{\R^{N}} \psi\left(\frac{y}{|x|}\right)A_{\nu}(y,D)u(y)
\,dy\right|^{q}\frac{|K(x,0)|^{q}}{|x|^{N-(N-\ell)q}}\,dx\right)^{1/q}    \\
& \lesssim \displaystyle{\left(\int_{\R^{N}}\left| \int_{B_{|x|/2}} 
\frac{|y|}{|x|}A_{\nu}(y,D)u(y)\,dy\right|^{q}\frac{1}{|x|^{N}}\,dx\right)^{1/q}}\\
&=\displaystyle{\left(\int_{\R^{N}}\left| \int_{B_{|x|/2}}  |y|A_{\nu}(y,D)u(y)\,dy\right|^{q}\frac{1}{|x|^{N+q}}\,dx\right)^{1/q}},
\end{align*}
where the inequality in the first line follows from \eqref{4.11}. The estimate in the second line is a consequence of \eqref{4.6} in Proposition \ref{prop1} with the choice 
$\varphi(y)\doteq \psi(y/|x|) f$ where, for fixed $x$, $f\in F$ is a unit vector chosen so that 
%\textcolor{blue}{Tiago: nao falta um modulo do lado esquerdo?}
\[
 \left|\int_{\R^{N}} \psi\left(\frac{y}{|x|}\right)f\cdot A_{\nu}(y,D)u(y)\,dy \right|=
\left| \int_{\R^{N}} \psi\left(\frac{y}{|x|}\right)A_{\nu}(y,D)u(y)\,dy\right|.
\] 
Then \eqref{k1} implies what we want after noticing  that, since all the derivatives of positive order of $\psi$ vanish at the origin,  
\[
|x|^j|D^j\varphi(y)|=\big|D^j\psi(y/|x|)\big|\le C\frac{|y|}{|x|},\quad 1\le j\le m.
\]
It follows that
\begin{align*}
J_{1}& \lesssim \int_{\R^{N}}|y|\,\,|A_{\nu}(y,D)u(y)|\left( \int_{\R^{N}\backslash B(0,2|y|)} \frac{1}{|x|^{N+q}}\,dx \right)^{1/q}dy \\
& \lesssim \int_{\R^{N}} |A_{\nu}(y,D)u(y)| dy. 
%& = \|A_{\nu}(x,D)u\|_{L^{1}} 
\end{align*}
Furthermore, the arguments in Lemma \ref{lemmafrac} may be used to show that 
$J_{2} \lesssim \|A_{\nu}(y,D)u\|_{L^{1}}$ so \eqref{4.5} is proved as we wished. 
 \Qed
\end{proof}

{The analog of inequality \eqref{main2.2a} for linear differential operators with variable coefficients is the following
\begin{proposition} \label{thm1.5}
Let $A(x,D)$ as before such that $\nu=N$. If $A(x,D)$ is elliptic and canceling there exist $B=B(0,\delta) \subset \Omega$ such that the a priori estimate
\begin{equation}\label{p6}
\int_{\R^{N}} \frac{|u(x)|}{|x|^{N}} \,dx \leq C \int_{\R^{N}}|A(x,D)u|| \omega(x)\,dx, \quad \omega(x)=|\log |x||
\end{equation} 
holds for some $C=C(B)>0$ and for all $u \in C_{c}^{\infty}(B;E)$ satisfying $u(0)=0$.
\end{proposition}
The proof is just a  repetition of the arguments in Theorem B and combines  \eqref{2.6} with a local version of \eqref{main2.2} given by
\begin{equation*}
\int_{\R^{N}} \frac{|\nabla u(x)|}{|x|^{N-1}} \omega(x) \,dx \leq C \int_{\R^{N}}|A_{\nu}(x,D)u|| \omega(x)\,dx, \quad u \in C_{c}^{\infty}(B;E). %\quad \omega(x)=|\log |x||
\end{equation*} 
}

%%%%%%%%%%%%%%%ss4.1
\subsection{Proof of Proposition \ref{prop1}} \label{sec4.1}
We will identify $E$ and $F$ with standard Euclidean spaces endowed with the usual inner product that will be denoted by a dot, for instance, we will write $f_1\cdot f_2$ for
 $f_1,f_2$ in $F$. If $k \in \mathcal{L}(F)$ is a linear transformation, $k^* \in \mathcal{L}(F)$ will denote the adjoint of $k$ with respect to the inner product.

Since $A(x,D)$ is elliptic and canceling, it follows from the proof of  \cite[Theorem 4.2]{HP3} that there exist a  linear differential operator 
$L(x,D):C^{\infty}(\Omega;F) \rightarrow C^{\infty}(\Omega;F)$ of order, say,  $m$
  such that
\begin{equation}\label{4.11}
\bigcap_{|\xi|=1}\sigma_{\nu}(A)(0,\xi)[E]=\bigcap_{|\xi|=1}Ker \;\sigma_{m}(L)(0,\xi).
\end{equation}
Since $A(x,D)$ is canceling at 
$0$, both members in \eqref{4.11} reduce to $\{0\}$ and, in particular, $L(x,D)$ is co-canceling at $0$, which means that
\[
\bigcap_{|\xi|=1}Ker \;\sigma_{m}(L)(0,\xi)=\{0\}.
\]
Writing 
$\displaystyle{L(x,D)=\sum_{|\alpha|\leq m}b_{\alpha}(x)\partial^{\alpha}}$, it follows from  \cite[Lemma 2.4]{HP3} that there exist a ball 
$B=B(0,r) \subset \Omega$ and functions $k_{\alpha} \in C^{\infty}(B; \mathcal{L}(F))$ such that 
\begin{equation}\label{ident1}
\sum_{|\alpha|=m}k_{\alpha}(x)b_{\alpha}(x)=I_{F}, \quad x \in B.
\end{equation} 
 Let $P: B \rightarrow \mathcal{L}(F)$ be given by 
 $\displaystyle{P(x)=\sum_{|\beta|=m}\frac{x^{\beta}}{\beta!}k_{\beta}^{\ast}(x)}$
and let $\alpha \in \Z^{N}_{+}$ satisfy  $|\alpha|=m$.
Hence, since  $\partial^{\alpha}(x^{\beta})=0$ for $|\beta|=m$ when $\alpha \neq \beta$, we have
 \begin{align*}
 \partial^{\alpha}  P(x)&=
\partial^{\alpha}\Bigg(\frac{x^{\alpha}}{\alpha!}k_{\alpha}^{\ast}(x)+   
\sum_{\substack{|\beta|=m \\ \beta \neq \alpha}}\frac{x^{\beta}}{\beta!}k_{\beta}^{\ast}(x) \Bigg)\\
%& =\partial^{\alpha}\left(\frac{x^{\alpha}}{\alpha!}k_{\alpha}^{\ast}(x)\right)+
%\partial^{\alpha} \sum_{\substack{|\beta|=m \\ \beta \neq \alpha}}\frac{x^{\beta}}{\beta!}\partial^{\alpha}k_{\beta}^{\ast}(x)\\
 &=k_{\alpha}^{\ast}(x)+ 
\sum_{\substack{|\beta|>0\\|\gamma|\le m,\,|\delta|=m }} c(\beta,\gamma,\delta)x^{\beta}\partial^{\gamma}k^{*}_{\delta}(x)\\% +\sum_{\substack{|\beta|=m \\ \beta \neq \alpha}}\frac{x^{\beta}}{\beta!}\partial^{\alpha}k_{\beta}^{\ast}(x)\\
& \doteq k_{\alpha}^{\ast}(x)+ R_{\alpha}(x).\
 \end{align*}
Thus, thanks to  \eqref{ident1},  the transpose $L_{m}^{*}(x,D)$ of $L_{m}(x,D)$ satisfies
 \[
L_{m}^{*}(x,D)(P(x))=\sum_{|\alpha|=m}b^{\ast}_{\alpha}(x)\partial^{\alpha}(P(x))=Id_{F}+R(x),
\]
 where $R(x)=\sum_{|\alpha|=m}b^{\ast}_{\alpha}(x)R_{\alpha}(x)$. 
Then, given $u \in C_{c}^{\infty}(B;E)$ and $\varphi \in C_{c}^{\infty}(B;F)$ we may write
$\varphi(x)=L_{m}^{*}(x,D)(P(x))\varphi(x)-R(x)\varphi(x)$ and 
\begin{equation}\label{4.13}
\begin{aligned}
\int_{R^{N}}&\varphi(x)\cdot A_{\nu}(x,D)u =I_{1}+I_{2}, \\
 I_{1}&= \int_{R^{N}} L^{*}_{m}(x,D)(P(x))\,\varphi \cdot A_{\nu}(x,D)u, \\
 I_{2}&=- \int_{R^{N}} R(x)\varphi(x)\cdot A_{\nu}(x,D)u. \\
\end{aligned}
\end{equation}
Writing  $T_{\varphi}(x)=L^{*}_{m}(x,D)(P(x))\,\varphi - L^{*}_{m}(x,D)(P\varphi)(x)$ we have
\[
I_{1}=\int_{\R^{N}}T_{\varphi}(x)  
\]
since $L_{m}(x,D)A_{\nu}(x,D)=0$. 
By Leibniz rule we have 
  \begin{equation}\nonumber
T_{\varphi}(x)=-  \sum_{|\alpha|=m}b_{\alpha}^{*}(x)\sum_{0<\gamma\le\alpha} \binom{\alpha}{\gamma}\partial^{\alpha-\gamma}P(x)\partial^{\gamma}\varphi(x)
\end{equation}  
where
\[
\partial^{\alpha-\gamma}P(x)=\sum_{|\beta|=m}\sum_{\eta \leq \alpha-\gamma} \left[\frac{1}{\beta!}\partial^{\alpha-\gamma-\eta}k_{\beta}^{*}(x) \right]\partial^{\eta}x^{\beta}.
\]
Now
 \begin{equation*}
 \partial^\eta(x^{\beta})= 
 \begin{cases}
\displaystyle(\beta!/(\beta-\eta)!)x^{\beta-\eta}& \text{if $\eta \leq \beta$},\\
0&\text{otherwise}.
\end{cases}
\end{equation*}
Therefore
$$|\partial^{\alpha-\gamma}P(x)|\lesssim \sum_{|\beta|=m}\sum_{\substack{\eta \leq \alpha-\gamma\\ \eta \leq \beta}} |x|^{|\beta-\eta|} \lesssim |x|^{|\gamma|},$$
for $x$ in the support of $u(x)$ and 
$$|\beta-\eta|=|\beta|-|\eta| \geq m-|\alpha-\gamma|\geq m -m +|\gamma|=|\gamma|.$$ 
Combining the previous estimates we conclude that
\begin{equation}\label{4.14}
|T_{\varphi}(x)|\leq C \sum_{j=1}^{m}|x|^{j}|D^{j}\varphi(x)|.
\end{equation}  
On the other hand, it follows from the definition of $R(x)$ that
\begin{equation}\label{4.15}
|R(x)|\leq C \sum_{j=1}^{m}|x|^{j}|D^{j}\varphi(x)|.
\end{equation}  
Then \eqref{4.13}, \eqref{4.14} and \eqref{4.15} imply \eqref{4.6} and 
Proposition \ref{prop1} is proved.

\Qed

%%%%%%%%%%%%%%%%%%%%%%%%%%%ss4.2

\subsection{{Applications to elliptic operators associated to systems of complex vector fields}} \label{sec4.2}
{We first present local versions of the classical Hardy-Sobolev inequalities for linear differential operators with variables coefficients.
\begin{proposition}\label{prop4.4}
Let $A(x,D)$ an elliptic linear differential operators of order $\nu$ with smooth complex coefficients in $\Omega \subset \R^{N}$, $N\ge2$, from $E$ to $F$. Consider $\ell \in \left\{ 1,...,\min(\nu,N-1)\right\}$ and $1\leq q<N/(N-\ell)$. Then we have the equivalents properties:
\begin{enumerate}
\item[(i)] for every $x_{0} \in \Omega$ there exist a ball $B=B(x_{0},r) \subset \Omega$ such that the a priori estimate
\begin{equation}\label{maineq3}
\left(\int_{\R^{N}}\frac{|D^{\nu-\ell}u(x)|^{q}}{|x-x_0|^{N-(N-\ell)q}}\,dx \right)^{1/q}\leq C \int_{\R^{N}}|A(x,D)u|dx, \  u \in C_{c}^{\infty}(B;E),
\end{equation} 
holds for some $C=C(B)>0$,  $u \in C_{c}^{\infty}(\R^{N};E)$;
\item[(ii)] $A(x,D)$ is canceling on $\Omega$.
\end{enumerate}
\end{proposition}
\begin{remark}
The limiting case of this proposition for $q=N/(N-\ell)$ and $\ell=1$, i.e., a characterization result for the Sobolev-Gagliardo-Nirenberg inequality, was obtained in \cite{HP3}.
\end{remark}
}
%

%\textcolor{blue}{
\begin{remark}
{It follows from a simple adaptation of Theorem \ref{thm2.1} that if  \eqref{maineq} holds for some $1<q<N/(N-\ell)$ then $A(x,D)$ must to be elliptic on $\Omega$.}
\end{remark}

The inequality \eqref{maineq3} at $x_{0}=0$ is a consequence of Proposition \ref{prop4.1} applying  $P_{\nu-\ell}:=D^{\nu-\ell}$ assuming $A(x,D)$ canceling. Clearly the argument may be extended for each $x_{0} \in \Omega$. The converse follows from a general method already used in \cite[Theorem 2.2]{HP3} that reduces the question to $A_{\nu}(x_{0},D)$ i.e. an homogenous linear differential operator with constant coefficients. %It will be presented here for sake of completeness. 
Assume that $\eqref{maineq3}$ holds for some $B=B(x_0,r)\subset\Omega$ and $C>0$. Decreasing $r>0$ conveniently to absorb the terms
$\|a_\alpha(\cdot) D^\alpha u\|_{L^1}$ with $|\alpha|<\nu$ and enlarging $C>0$, we may assume that for all $u \in C_{c}^{\infty}(B;E)$
\begin{equation}\label{maineq2}
\left(\int_{\R^{N}}\frac{|D^{\nu-\ell}u(x)|^{q}}{|x-x_0|^{N-(N-\ell)q}}\,dx \right)^{1/q}\leq C \int_{\R^{N}}|A_{\nu}(x,D)u|\,dx, 
\end{equation} 
where $A_{\nu}(x,D)=\sum_{|\alpha|=\nu}a_{\alpha}(x)D^{\alpha}$ is the principal part of $A(x,D)$. Choose any $u\in\ccinf(\R^N;E)$ and set $u_{\eps}(x)\doteq u(\eps^{-1}(x-x_0))$. 
For sufficiently small $\eps>0$,  $supp (u_\eps) \subset B$ and  \eqref{maineq2} holds for $u=u_\eps$. 
Canceling powers of $\eps$ on both sides and letting $\eps\searrow0$ we obtain
\[
\left(\int_{\R^{N}}\frac{|D^{\nu-\ell}u(x)|^{q}}{|x-x_0|^{N-(N-\ell)q}}\,dx \right)^{1/q} \le   C \int_{\R^{N}}|A_{\nu}(x_{0},D)u|dx ,
\]
for any $ u \in C_{c}^{\infty}(\R^N;E)$.
Applying  \cite[Proposition 4.1]{BVS} %to the homogeneous operator with constant coefficients $A_\nu(x_0,D)$ 
we conclude that $A_\nu(x_0,D)$ is canceling, i.e., $A(x,D)$ is canceling at $x_0$. 
\Qed

We now consider applications of Proposition \ref{prop4.4} to a special class of canceling elliptic differential operators associated to complex vector fields recently studied in the works \cite{HP1}, \cite{HP2} and \cite{HP3}. Consider $n$ complex vector fields $L_{1},...,L_{n}$, $n \geq 1$, with smooth coefficients defined on a neighborhood of the origin $\Omega$ in $\R^{N}$,   
 $N\geq 2$. 
 
 \begin{corollary}\label{order1}
 If the system of vector fields $L_{1},...,L_{n}$, $n\geq 2$, with smooth complex coefficients, is linearly independent and elliptic then every point $x_{0} \in \Omega$ is contained in a ball $B=B(x_{0},r) \subset \Omega$ such that 
 \begin{equation}\label{ineqorder1}
 \int_{\R^{N}} \frac{|u(x)|}{|x|}\,dx\leq C \sum_{j=1}^{n}\|L_{j}u\|_{L^{1}}, \quad u \in C_{c}^{\infty}(B),
 \end{equation}
 for some $C=C(B)>0$. 
 \end{corollary}
 
 The ellipticity means that for any real 1-form $\omega$ satisfying $\omega(L_{j})=0$ for $j=1,...,n$ implies $\omega=0$.
 The proof follows from applying Proposition \ref{prop4.4} 
to  the operator $A(x,D):=\nabla_{\L}:C^{\infty}(\Omega) \rightarrow C^{\infty}(\Omega)$ defined by $\nabla_{\L}u\,{\doteq}\,(L_{1}u,...,L_{n}u)$ which is an elliptic and canceling operator (see \cite[Lemma 4.1]{HP3}). In the particular setting $L_{j}=\partial_{x_{j}}$ and $n=N$ the estimate \eqref{ineqorder1} expresses  a local version of the classical Hardy inequality. % for $p=1$. 
 
 \begin{remark}
   Clearly the estimate \eqref{ineqorder1} implies    
  \begin{equation}\nonumber
\|u\|_{L^{1}} \leq C(B) \sum_{j=1}^{n}\|L_{j}u\|_{L^{1}}, \quad u \in C_{c}^{\infty}(B),
 \end{equation}
 where $C(B) \rightarrow 0$ when $r \rightarrow 0$.
 \end{remark}

Next we describe a version of Corollary \ref{order1} in the setup of complexes or pseudo complexes $\left\{d_{\L,k} \right\}_{k}$ associated to the system $\L=\left\{L_{1},...,L_{n}\right\}$. Let  $C^{\infty}(\Omega;\Lambda^{k}\R^{N})$ be the space of $k$-forms on $\R^{N}$, $0\leq k \leq n$, with complex smooth coefficients defined on $\Omega$. Each  $f \in C^{\infty}(\Omega;\Lambda^{k}\R^{N})$ may be written as  
$
\displaystyle{f=\sum_{|I|=k}f_{I}dx_{I}}, \quad dx_{I}=dx_{i_1}\wedge\cdots\wedge dx_{i_k},
$
where $f_{I}\in C^{\infty}(\Omega)$ and $I=\left\{i_{1},...,i_{k}\right\}$ is a set of strictly increasing indexes with $i_{l}\in \left\{1,...,n\right\}$,  $l=1,...,k$.  Consider the differential operators 
\[
d_{\L,k}: C^{\infty}(\Omega;\Lambda^{k}\R^{N})\rightarrow C^{\infty}(\Omega;\Lambda^{k+1}\R^{N})
\]
 given by
%\begin{equation}\nonumber
$d_{\L,0}f=\sum_{j=1}^{n}(L_{j}f)dx_{j}$ for $f \in C^{\infty}(\Omega)$
%\end{equation}  
and for  $1\le k\le n-1$,
\begin{equation}\label{4.5a}
d_{\L,k}f=\sum_{|I|=k}(d_{\L,0}f_{I})dx_{I}=\sum_{|I|=k}\sum_{j=1}^{n}(L_{j}f_{I})dx_{j} \wedge dx_{I}.
\end{equation}  
%Note that, in general, the complex property $d_{\L,k+1}d_{\L,k}=0$ does not hold. 
We also define the dual $d^{*}_{\L,k}: C^{\infty}(\Omega;\Lambda^{k+1}\R^{N})\rightarrow 
C^{\infty}(\Omega;\Lambda^{k}\R^{N})$, for $0 \leq k \leq n-1$, determined by 
\[
\int d_{\L,k}u\cdot \overline{v}\, dx=\int u \cdot \overline{d^{*}_{\L,k}v}\,dx,
\ u \in C_{c}^{\infty}(\Omega;\Lambda^{k}\R^{N}),\ v \in 
C_{c}^{\infty}(\Omega;\Lambda^{k+1}\R^{N}),
\]
where the dot indicates the standard pairing on forms of the same degree.  In general, the chain operators $\left\{ d_{\L,k}\right\}_{k}$ do not define a complex, however the property $d_{\L,k+1}\circ d_{\L,k}=0$ is always satisfied  in the sense of principal symbols. We will refer to $(d_{\L,k},C^{\infty}(\Omega;\Lambda^{k}\R^{N}))$ as the pseudo-complex $\{d_{\L}\}$ associated with $\L$ on $\Omega$. In the 
involutive 
situation the chain operators  $\left\{ d_{\L,k}\right\}_{k}$ define a true complex associated to the structure $\L$. This structure is precisely the de Rham complex for the special case $L_{j}=\partial_{x_{j}}$ and $n=N$ (see \cite{HP2}). 

 \begin{corollary}\label{order11}
 Assume that system of vector fields $L_{1},...,L_{n}$, $n\geq 2$ is linearly independent and elliptic and that $k$ is neither $1$ nor $n-1$. Then every point $x_{0} \in \Omega$ is contained in a ball $B=B(x_{0},r) \subset \Omega$ such that 
 \begin{equation}\nonumber
 \sum_{|I|=k}\int_{\R^{N}} \frac{|u_{I}(x)|}{|x|}\,dx\leq C \left( \|d_{\L,k}u\|_{L^{1}}+\|d^{*}_{\L,k}u\|_{L^{1}} \right), \  u \in C_{c}^{\infty}(B;\Lambda^{k}\R^{N}),
 \end{equation}
 for some $C=C(B)>0$.  
 \end{corollary}

The validity of this inequality is due to fact that the elliptic operator 
$$(d_{\L,k},d^{*}_{\L,k}):C_{c}^{\infty}(\Omega; \Lambda^{k}\R^{N}) \rightarrow C_{c}^{\infty}(\Omega;\Lambda^{k+1}\R^{N}) \times C_{c}^{\infty}(\Omega;\Lambda^{k-1}\R^{N})$$
is canceling if only if $k$ is neither 1 nor $n-1$ (\cite[p.15]{HP3}).
An extension of this result may be considered for higher order pseudo-complexes (order $2m+1$) defined by $S_{\L,k,2m+1}{\doteq}d_{\L,k}(d^{\ast}_{\L,k}d_{\L,k})^{m}$ and $m \in \N$ (\cite[p.16]{HP3}).

%%%%%%%%%%%%%%%%%%%%%%%%%%%%%%%%%%%%%%%%%%%%%%%%%%%%%%%%%%%%%%%%%%%%%%%%%%%%%%%%%%%

\end{document}